\documentclass[a4paper, reqno, 10pt]{amsart}
\usepackage{amssymb,latexsym, mathdots}
\usepackage{graphicx}
\usepackage{color}

\newtheorem{theorem}{Theorem}
\newtheorem{definition}{Definition}

\newtheorem{corollary}{Corollary}
\newtheorem{lemma}{Lemma}

\begin{document}
\title[Hook Shape Lattice Path Matroid Polytopes]{Toric $g$-polynomials of hook shape lattice Path Matroid Polytopes and product of simplices}

\author{Sen-Peng Eu}
\address{Department of Mathematics \\
National Taiwan Normal University \\
Taipei, Taiwan 116, ROC} \email[Sen-Peng Eu]{speu@math.ntnu.edu.tw}

\author{Yuan-Hsun Lo}
\address{School of Mathematical Science \\
Xiamen University \\
Xiamen, 361005, PRC} \email[Yuan-Hsun Lo]{yhlo0830@gmail.com}

\author{Ya-Lun Tsai}
\address{Department of Applied Mathematics \\
National Chung Hsing University \\
Taichung, Taiwan 402, ROC} \email[Ya-Lun Tsai]
{yltsai@nchu.edu.tw}

\date{\today}

\begin{abstract}
It is known that a lattice path matroid polytope can be associated
with two given noncrossing lattice paths on
$\mathbb{Z}\times\mathbb{Z}$ with the same end points. In this short
note we give explicit formulae for the $f$-vector, toric $f$- and
 $g$-polynomials of a lattice path matroid polytope when two
boundary paths enclose a hook shape.
\end{abstract}



\thanks{Partially supported by National Science Council, Taiwan under grants MOST 104-2115-M-003-014-MY3 (S.-P. Eu) and MOST 104-2115-M-005-004 (Y.-L. Tsai).}


\maketitle


\section{Introduction}
Lattice path matroid polytopes are those polytopes defined from a family of lattice paths bounded by two nonintersecting paths with the same endpoints.
In this paper we characterize the shape and compute its $f$-vector, toric $f$- and $g$-polynomial when the region
bounded by two paths is a hook shape.

Though only the hook shape cases are considered, the explicit
formula of the toric $g$-polynomial (see Theorem 5) obtained is
surprisingly neat and even in such simple cases the computations are
not trivial. As we are not able to find a direct proof in the
literature, we are happy to write it down `from scratch'. In the
following we give preliminary background needed for the rest.

\subsection{Lattice path matroid}
A \emph{matroid} $M$ is a finite collection $\mathcal{S}$ of subsets, called \emph{independent sets}, of $[n]:=\{1,2,\dots, n\}$ satisfying the following conditions:
\begin{enumerate}
\item $\emptyset \in \mathcal{S}$.
\item If $A\in \mathcal{S}$ and $B\subset A$, then $B\in \mathcal{S}$.
\item If $A,B \in \mathcal{S}$ and $|A|=|B|+1$, then there exists $x\in A\backslash B$ such that $B\cup \{x\}\in S$.
\end{enumerate}

A \emph{base} of a matroid $M$ is defined to be a set of a maximal independent sets.
We denote the set of bases by $\mathcal{B}$.

In 2003, Bonin et al.~\cite{BMN1} proposed the notion of lattice path matroid, while almost at the same time Ardilla~\cite{A} also independently proposed the notion of \emph{Catalan matroid}, which turned out to be a special case of lattice path matroid.
Fix two noncrossing lattice paths $\pi_1, \pi_2$ (with $\pi_1$ never going below $\pi_2$) on $\mathbb{Z}\times \mathbb{Z}$ from $(0,0)$ to $(s,t)$ using east and north steps, Bonin et al. proved the following theorem:

\begin{theorem}[\cite{BMN1}]
Each path $\sigma$ of the set of lattice paths from $(0,0)$ to
$(s,t)$ staying the region bounded by $\pi_1$ and $\pi_2$
corresponds to a base of a matroid $M_{\pi_1,\pi_2}$. In fact, set
$\sigma=\sigma_1 \sigma_2\dots \sigma_{s+t}$ with $\sigma_i$ being
the up step $\mathsf{N}=(0,1)$ or east step $\mathsf{E}=(1,0)$, then
the set of those $\mathsf{N}$ steps in $\sigma$ is a base of
$M_{\pi_1,\pi_2}$.
\end{theorem}
A matroid is a \emph{lattice path matroid} if it is isomorphic to
$M_{\pi_1,\pi_2}$. Readers may refer to ~\cite{A, B, BMN1, BMN2} for
more information on lattice path matroids.

For example, fix $\pi_1=\mathsf{NENE}$ and $\pi_2=\mathsf{EENN}$, there are five lattice paths bounded by $\pi_1, \pi_2$, see Figure 1.
The path $\mathsf{NENE}$ corresponds to the base $13$ (shorthand of $\{1,3\}$); and $\mathsf{NEEN}$ to $14$, etc.
The set of bases of $M_{\pi_1,\pi_2}$ is $\mathcal{B}=\{13, 14, 23,24,34\}.$

\begin{figure}[h]
\includegraphics[width=3in]{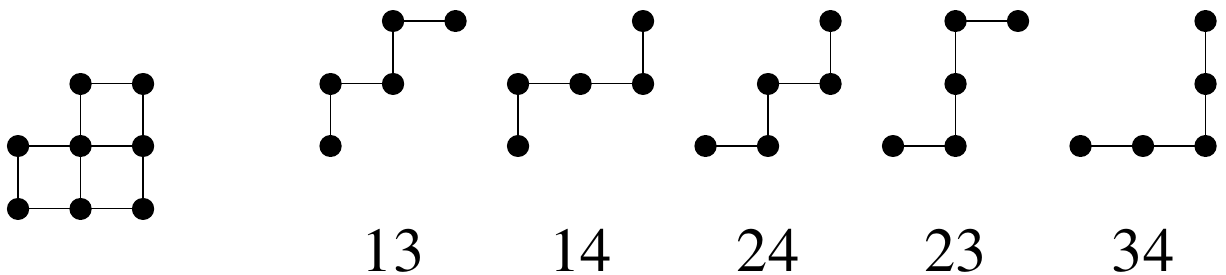}
\caption{Bases of the matroid $M_{\pi_1,\pi_2}$ with $\pi_1=\mathsf{NENE}$ and $\pi_2=\mathsf{EENN}$}
\end{figure}

%
%
%
%

\subsection{Lattice path matroid polytope}
From a matroid $M$ one can define a \emph{matroid polytope} $P_M$
in the following ways. Let $B=\{ \sigma_1,\sigma_2, \dots ,\sigma_r
\} \in \mathcal{B}$ and the \emph{incidence vector} $\mathbf{e}_B$
of $B$ by $\mathbf{e}_B:=\sum_{i=1}^r\mathbf{e}_{\sigma_i}$, where
$\mathbf{e}_j$ is the $j$-th standard unit vector of $\mathbb{R}^n$.
We define $$P_M:=\mbox{conv}\{\mathbf{e}_B: B\in \mathcal{B}\},$$
the convex hull of all incidence vectors.

Hence one can consider the \emph{lattice path matroid polytope}
defined from a lattice path matroid. For instance, take
$M=M_{\pi_1,\pi_2}$ in the above example. For $B=13$ (shorthand for
$\{1,3\}$) we have
$\mathbf{e}_B=\mathbf{e}_1+\mathbf{e}_3=(1,0,1,0)$. Hence the
desired polytope is
$$P_M=\mbox{conv}\{(0,0,1,1), (0,1,0,1), (0,1,1,0), (1,0,0,1), (1,0,1,0)\},$$
which is a pyramid as shown in Figure 2.

\begin{figure}[h]
\includegraphics[width=1.7in]{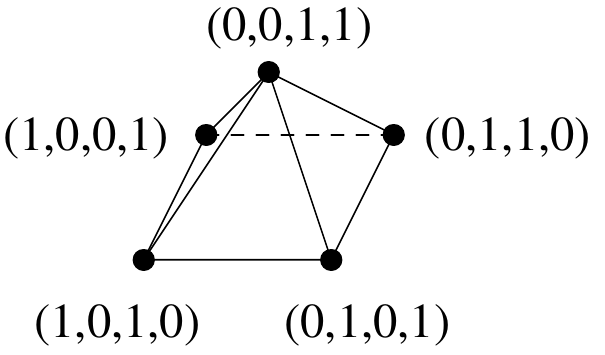}
\caption{The lattice path matroid polytope $P_{2,2}$}
\end{figure}

To our knowledge there are only few results on lattice path matroid polytope~\cite{B, K}.
In this paper we focus on the cases that the region bounded by $\pi_1$ and $\pi_2$ is a hook shape.
We denote such a  polytope $P_{\alpha, \beta}$ if $\pi_2$ goes first with $\alpha$ straight $\mathsf{E}$ steps then $\beta$ $\mathsf{N}$ steps.
That is, the hook shape corresponds to the partition $\lambda=(\alpha,1,1,\dots 1)$ with $\beta-1$ one's.
Our example is the $P_{2,2}$.

For $0\le k\le n$, let $f_k$ be the number of $k$-dimensional faces of the polytope $P$ and  call $(f_0,f_1,\dots , f_n)$ the \emph{$f$-vector} of $P$.
The first result is to characterize the shape of $P_{\alpha, \beta}$ and compute the $f$-vector.
Denote the $n$-dim simplex by $\Delta_n$.

\begin{theorem}\label{main1}
We have
\begin{enumerate}
\item The $P_{\alpha, \beta} \subset \mathbb{R}^{\alpha+\beta-1}$ is a pyramid with the basis of the Cartesian product of the simplices $\Delta_{\alpha-1}\subset \mathbb{R}^{\alpha-1}$ and $\Delta_{\beta-1} \subset \mathbb{R}^{\beta-1}$.
\item Let $(f_{0},\dots,f_{\alpha+\beta-1})$ denote the $f$-vector of the $P_{\alpha, \beta}$.
Let $r_{-1}:=1$ and
$$ r_{i}:=\sum_{k=1}^{i+1}{\alpha \choose k}{\beta \choose i+2-k},$$
then for $0\le i\le \alpha+\beta-1$ we have $f_{i}=r_{i}+r_{i-1}$.
\item All edges of $P_{\alpha, \beta}$ has the length $\sqrt{2}$.
\item The diameter of $P_{\alpha, \beta}$ is $2$.
\end{enumerate}
\end{theorem}

\subsection{Toric $h$-vector}
From the $f$-vector of a $d$-dimensional polytope $P$ one can
calculate the $h$-vector by $$h_i=\sum_{j=0}^i(-1)^{i-j}{d-j \choose
i-j}f_{j-1},$$ which plays an important role in analyzing the
polytope. It is well known that when $P$ is simplicial, one has a
nice symmetric property.

\begin{theorem} [Dehn-Sommerville equations, see~\cite{Z}]
The $h$-vector of the boundary of a simplical $d$-polytope $P$ satisfies $$h_i=h_{d-i}.$$
\end{theorem}

However the lattice path matroid polytope is usually not simplical, as in the case of $P_{2,2}$, and the Dehn-Sommerville equations are not applicable.
Motivated from algebraic geometry, for any finite graded poset with $\hat{0}$ and $\hat{1}$ Stanley introduced the \emph{toric $h$-vector} and \emph{toric $g$-vector}.
The motivation for introducing toric $h$-vectors is to correspond to the Betti numbers of the intersection cohomology of toric varieties associated to rational polytopes.
Readers may see~\cite{S} for more information.

We begin by defining the toric $f$-polynomial and toric $g$-polynomial.
Let $\mathcal{P}$ be the face lattice of a convex polytope $P$ in $\mathbb{R}^{n}$.
Hence $\mathcal{P}$ is graded of rank $n+1$ with the rank function
$\rho$.
Let $\widetilde{\mathcal{P}}$ be the set of all intervals
$[\emptyset,y]$ for all $y$ in $\mathcal{P}$, ordered by inclusion.
The map $\mathcal{P}\to \widetilde{\mathcal{P}}$ by $y \mapsto
[\emptyset,y]$ is an isomorphism of posets and hence
$\widetilde{\mathcal{P}}$ inherits the same rank function $\rho$.

\begin{definition}\label{def}
Given an finite graded poset $\mathcal{P}$, the toric $f$-polynomial
$f(\mathcal{P},x)$ and toric $g$-polynomial $g(\mathcal{P},x)$ are
defined inductively as follows.
\begin{enumerate}
\item $f(\mathbf{1},x)=g(\mathbf{1},x)=1$, where $\mathbf{1}$ is the poset of the single vertex.
\item If rank of $\mathcal{P}$ is $n+1>0$, then $f(\mathcal{P},x)$ has degree $n$. Suppose $f(\mathcal{P},x)=h_{0}+h_{1}x+\cdots+h_{n}x^{n}$,
then define
$$g(\mathcal{P},x):=h_{0}+(h_{1}-h_{0})x+\cdots+(h_{m}-h_{m-1})x^{m},$$
where $m=\lfloor \frac{n}{2}\rfloor$.
\item If rank of $\mathcal{P}$ is $n+1>0$, then define
$$f(\mathcal{P},x):=\sum_{\mathcal{Q} \in \widetilde{\mathcal{P}}, \mathcal{Q} \ne \mathcal{P}}g(\mathcal{Q},x)(x-1)^{n-\rho(\mathcal{Q})}.$$
 \end{enumerate}
\end{definition}

We call $(h_{0},h_{1},\dots,h_{n})$ the \emph{toric $h$-vector} and
$(g_{0},g_{1},\dots,g_{n})$ the \emph{toric $g$-vector} of
$\mathcal{P}$. The following result, extending the Dehn-Sommerville
equations, is the core result about toric $h$-vector which states
that the entries are symmetric if $\mathcal{P}$ is Eulerian.

\begin{theorem}[\cite{S}]\label{thm:toric_sym}
Let $\mathcal{P}$ be Eulerian posets of rank $n+1$, Then have
$$h_{i}=h_{n-i}.$$
\end{theorem}

Since the face lattice of a convex polytope, say $P_{\alpha,\beta}$, is a Eulerian poset, by the theorem the entries of its toric $h$-vector are symmetric.
Also, note that in this case $f(\mathcal{P},x)$ is uniquely determined by $g(\mathcal{P},x)$.

Given $\mathcal{P}$, it is usually not easy to compute explicitly
the toric $h$- or toric $g$-vector. To our knowledge there are very
few examples~\cite{S} (Section 3.16 and Exercises 3.176, 3.177). As
Stanley noted in~\cite{S}, $f(\mathcal{P},x)$ `seems to be an
exceedingly subtle invariant of $\mathcal{P}$'.

Our main theorem of this paper is a surprising neat formula for the toric $g$-vector of (the face lattice of) $P_{\alpha, \beta}$.

\begin{theorem}\label{main2}
Let $1 \le \beta \le \alpha$, and $g_{\alpha,\beta}(x):=g(P_{\alpha,
\beta},x)$ be the toric $g$-polynomial of the $P_{\alpha, \beta}$.
We have
$$ g_{\alpha,\beta}(x)=\sum_{k=0}^{\beta-1}{\alpha-1\choose k} {\beta-1\choose k}x^{k}. $$
\end{theorem}

The toric $h$-vector can be derived directly from Definition~\ref{def}, Theorem~\ref{thm:toric_sym} and Theorem~\ref{main2}.

\begin{corollary}\label{cor:toric_f}
Let $1 \le \beta \le \alpha$, and $f_{\alpha,\beta}(x):=f(P_{\alpha,\beta},x)$ be the toric $f$-polynomial of the $P_{\alpha, \beta}$.
We have
$$f_{\alpha,\beta}(x)=\sum_{k=0}^{\beta-1}\mathbf{S}_k x^{k} + \mathbf{S}_{\beta-1}\sum_{k=\beta}^{\alpha-1}x^k + \sum_{k=\alpha}^{\alpha+\beta-1}\mathbf{S}_{\alpha+\beta-1-k}x^k, $$
where $$\mathbf{S}_{\ell}:=\sum_{k=0}^{\ell}{\alpha-1\choose k}{\beta-1\choose k}$$ for $0<\ell\leq\beta-1$.
\end{corollary}

For example, for $P_{2, 2}$, by Theorem 1 we can compute $(a_{-1},a_0,a_1,a_2,a_3)=(1,4,4,1,0)$ and therefore
$(f_0,f_1,f_2,f_3)=(5,8,5,1)$, as seen in Figure 2.
By Theorem 2 the toric $g$-polynomial is
$$g(P_{2,2},x)=\sum_{k=0}^1{1\choose k}{1\choose k}x^k=1+x$$
and the toric $f$-polynomial is $f(P_{2,2},x)=1+2x+2x^2+x^3$ by
Corollary~\ref{cor:toric_f}. Both can be computed directly (and
tediously) by definition and we omit the details.

\medskip
The rest of the paper is organized as follows. In Section 2 we
describe the shape of $P_{\alpha,\beta}$ and compute its $f$-vector.
Section 3 collect some preliminary facts needed for computing the
toric $g$-polynomial. The computation is finished in Section $4$.

\section{Shape and $f$-vector} In this section we prove Theorem 2, characterizing the shape of $P_{\alpha, \beta}$ and compute the $f$-vector.
Without loss of generality assume that $\alpha \ge \beta \ge 1$ and
let $m=\alpha-1, n=\beta-1$ for simplicity.
%

\medskip
\noindent \textit{Proof of Theorem 2 }: (1) As $P_{\alpha, \beta}$
is induced from the the lattice paths that goes from $(0,0)$ to
$(\alpha,\beta)$ and remains in the hook shape, we can denote the
incidence vectors in the following way. Let $\mathbf{e}_{k}\in
\mathbb{R}^{\alpha+\beta}$ be the elementary unit vector with the
only nonzero entry $1$ in its $k$-th coordinate. Then the incidence
vectors are $\mathbf{v}_{0,1}=(0,\dots,0,1,\dots,1)\in
\mathbb{R}^{\alpha+\beta}$ with the first $\alpha$ coordinates $0$'s
and the rest $\beta$ coordinates $1$'s, and
$\mathbf{v}_{i,j}=\mathbf{v}_{0,1}+\mathbf{e}_{i}-\mathbf{e}_{\alpha+j}$
 with $1\leq i \leq
\alpha$ and $1 \le j \le \beta$. We are to consider the polytope
$$P_{\alpha, \beta}:=\mbox{conv}\{\mathbf{v}_{0,1}, \mathbf{v}_{i,j} : 1\leq i \leq
\alpha, 1\leq j \leq \beta \}.$$

First we prove that $P_{\alpha, \beta}$ is a pyramid. It is clear
that $P_{\alpha, \beta} \subseteq \mathbb{R}^{\alpha+\beta-1}$ since
it is in the hyperplane $H_{\alpha+\beta-1}:=\{ (x_1,\dots x_\alpha,
y_1, \dots
y_\beta):x_{1}+\cdots+x_{\alpha}+y_{1}+\cdots+y_{\beta}=\beta\}$.
Let $$B:=\mbox{conv}\{\mathbf{v}_{i,j} : 1\leq i \leq \alpha, 1\leq
j \leq \beta \}, $$ then similarly one has $B\subseteq
H_{\alpha+\beta-2}:=\{ (x_1,\dots x_\alpha, y_1, \dots
y_\beta):x_{1}+\cdots+x_{\alpha}=1,
y_{1}+\cdots+y_{\beta}=\beta-1\}$. Note that we can also write
$$P_{\alpha, \beta}= \left\{\sum_{i,j} a_{i,j} \mathbf{v}_{i,j} + a_{0,1}\mathbf{v}_{0,1}: 0\le
a_{i,j},a_{0,1} \le  1, \sum_{i,j} a_{i,j}+a_{0,1}=1\right\}$$ and
$$B=\left\{ \sum_{i,j} a_{i,j} \mathbf{v}_{i,j} : 0\leq a_{i,j}\leq  1,  \sum_{i,j} a_{i,j}=1\right\}.$$
It is clear that the point $\mathbf{v}_{0,1}$ is not on
$H_{\alpha+\beta-2}$ and points in $B$ are all on
$H_{\alpha+\beta-2}$. Therefore, $P_{\alpha, \beta}$ is a pyramid
with the base $B$.

Next, we show that $B=\Delta_m\times \Delta_n$, where $\Delta_m $ and $\Delta_n$ denote an $m$-simplex and an $n$-simplex respectively.
Let $\mathbf{v}_{i}=(0,\dots,0,1,0,\dots,0), 1\leq i \leq \alpha$ be the unit vector in $\mathbb{R}^{\alpha}$ with the only non zero element $1$ in the $i$-th coordinate and $\mathbf{u}_{j}=(1,\dots,1,0,1,\dots,1), 1\leq j \leq \beta$ be the vectors in $\mathbb{R}^{\beta}$ with all coordinates $1$'s but the $j$-th coordinate $0$.
It is clear that
$$\Delta_m=\mbox{conv}(\mathbf{v}_{i})=\left\{\sum_{i=1}^{\alpha} a_{i}\mathbf{v}_{i}: 0\leq a_{i} \leq 1,  \sum_{i=1}^{\alpha} a_{i}=1\right\} \subseteq \left\{ ( x_{1}, \cdots, x_{\alpha}) : x_{1}+\cdots+x_{\alpha}=1 \right\}$$
in $\mathbb{R}^{m}(=\mathbb{R}^{\alpha-1})$,
and
$$\Delta_n=\mbox{conv}(\mathbf{u}_{j})=\left\{\sum_{j=1}^\beta b_{j}\mathbf{u}_{j} : 0\leq b_{j} \leq 1, \sum_{j=1}^\beta b_{j}=1\right \} \subseteq \{ (y_1, \dots , y_\beta): y_{1}+\cdots+y_{\beta}=\beta-1 \}$$
in $\mathbb{R}^{n}(=\mathbb{R}^{\beta-1})$.

If $\mathbf{w} \in \Delta_m\times \Delta_n$, then
$\mathbf{w}=(\mathbf{v},\mathbf{u})$, where
$\mathbf{v}=\sum_{i=1}^{\alpha} a_{i}\mathbf{v}_{i}$ and
$\mathbf{u}=\sum_{j=1}^{\beta} b_{j}\mathbf{u}_{j}$ for some $0\leq
a_{i}, b_{j} \leq 1$ with $\sum_{i=1}^{\alpha}
a_{i}=\sum_{j=1}^{\beta} b_{j}=1 $. Therefore
$$\mathbf{w}=\left(\sum_{i=1}^\alpha a_{i}\mathbf{v}_{i},\sum_{j=1}^\beta b_{j}\mathbf{u}_{j}\right)=\sum_{1\le i \le \alpha,\, 1\le j \le \beta} a_{i}b_{j}\mathbf{v}_{i,j},$$
as $\mathbf{v}_{i,j}=(\mathbf{v}_{i},\mathbf{u}_{j})$. It is clear
that $0\leq a_{i}b_{j} \leq 1$, and $\sum_{1\le i \le \alpha,\, 1\le
j \le \beta} a_{i}b_{j}=1$. So, $\mathbf{w}\in B$.

On the other hand, if $\mathbf{w}\in B$, then $\mathbf{w}=\sum_{1\le
i \le \alpha,\,1\le j \le \beta} c_{i,j}\mathbf{v}_{i,j}$, where $
0\leq c_{i,j} \leq 1$, and $\sum_{1\le i \le \alpha,\, 1\le j \le
\beta} c_{i,j}=1$. Since
$(\mathbf{v}_{i},\mathbf{u}_{j})=\mathbf{v}_{i,j}$, we get
$$
\mathbf{w} = \sum_{1\le i \le \alpha,\, 1\le
j \le \beta} c_{i,j}(\mathbf{v}_{i},\mathbf{u}_{j})
=\left(\sum_{i}a_{i}\mathbf{v}_{i},\sum_{j}b_{j}\mathbf{u}_{j}\right),
$$
with $a_{i}=\sum_{j=1}^{\beta} c_{i,j}$,
$b_{j}=\sum_{i=1}^{\alpha}c_{i,j}$. It is clear that $0\le a_i,
b_j\le 1$ and $\sum_{i=1}^\alpha a_{i}=\sum_{j=1}^\beta b_{j}=1$.
Therefore, $\mathbf{w} \in \Delta_m\times \Delta_n$.

(2) Now the $f$-vector is easy to compute.
Since $B=\Delta_m\times\Delta_n$, it has $$r_{i}:=\sum_{k=1}^{i+1}{\alpha \choose k}{\beta \choose i+2-k}$$ faces of dimension $i$.
Since $P_{\alpha, \beta}$ is a pyramid over $B$, it then has $r_i+r_{i-1}$ faces of dimension $i$ and we are done. Note the number $r_{-1}:=1$ counts the unique vertex not on the base, i.e., the apex.

(3) From $(2)$ we know there are $\alpha\beta+1$ vertices in
$P_{\alpha, \beta}$. Hence these vertices are exactly
$\mathbf{v}_{0,1}$ and $\mathbf{v}_{i,j}$, $1\le i\le \alpha$, $1\le
j\le \beta$. Let $a,b$ be the endpoints of an edge in
$P_{\alpha,\beta}$. There are three cases: (i) $a=\mathbf{v}_{i,j}$,
$b=\mathbf{v}_{i,j'}$ for some $i$ and $j\neq j'$, (ii)
$a=\mathbf{v}_{i,j}$, $b=\mathbf{v}_{i',j}$ for some $i\neq i'$ and
$j$, and (iii) $a=\mathbf{v}_{0,1}$, $b=\mathbf{v}_{i,j}$ for some
$i,j$. In either case, the distance of $a,b$ is $\sqrt{2}$.

(4) The diameter of $B=\Delta_m\times \Delta_n$ is clearly $2$ and also the pyramid having $B$ as its basis.

\qed

%

%

\section{Toric $h$-vector}
We embark on deriving the toric $g$-vector $g_{\alpha, \beta}(x)$.
Recall that $m=\alpha-1$ and $n=\beta-1$. In the following we call
$\Delta_m \times \Delta_n$ the \emph{reduced $(m,n)$ polytope}.  We
denote its toric $f$- and toric $g$-polynomial by
$\widetilde{f}_{m,n}(x)$ and $\widetilde{g}_{m,n}(x)$ respectively.

Our strategy of proof is as follows.
In this section we prove that $g_{\alpha,\beta}(x)$ equals to $\widetilde{g}_{m,n}(x)$.
The second step (Section 4) is to compute $\widetilde{g}_{m,n}(x)$.

\subsection{From $g_{\alpha,\beta}$  to $\widetilde{g}_{m,n}$}
For the reduced $(m,n)$ polytope with $m\ge n\ge 1$, its corresponding region is made of horizontal $m$ boxes touched at a corner by $n$ vertical boxes, see Figure 3(a) for the region corresponding to the reduced $(4,3)$ polytope.
The $m+1$ paths going from $(0,0)$ to $(m,1)$ corresponds to the $m+1$ vertices of an $m$-simplex, and the $n+1$ paths going from $(m,1)$ to $(m+1,n+1)=(\alpha,\beta)$ corresponds to the $n+1$ vertices of an $n$-simplex.
Also, by symmetry we have $\widetilde{g}_{m,n}=\widetilde{g}_{n,m},\widetilde{f}_{m,n}=\widetilde{f}_{n,m}$ and $g_{\alpha,\beta}=g_{\beta,\alpha},f_{\alpha,\beta}=f_{\beta,\alpha}$.

\begin{figure}[h]
\includegraphics[width=3in]{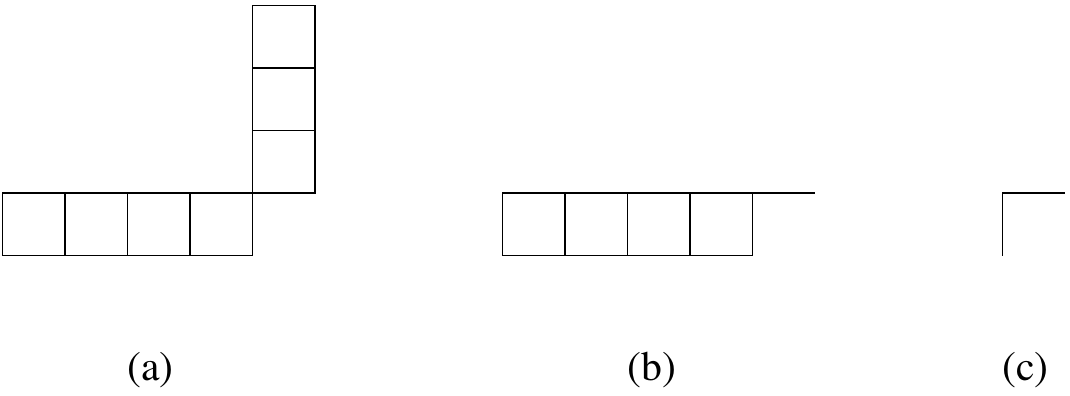}
\caption{Regions for reduced polytope}
\end{figure}

Some initial cases need to be clarified. For $\alpha>\beta =1$
($m>n=0$), The reduced $(m,0)$ polytope is induced from the bounded
region as illustrated in Figure 3(b) for the reduced $(4,0)$
polytope. The $m+1$ paths from $(0,0)$ to $(m,1)$ correspond to the
$m+1$ vertices of an $m$-simplex and the path from $(m,1)$ to
$(m+1,1)$ contributes nothing. Hence the reduced $(m,0)$ polytope is
an $m$-simplex and we have $\widetilde{g}_{m,0}=1$, see~\cite{S}.
Similarly, the reduce $(0,n)$ polytope is the $n$-simplex and
$\widetilde{g}_{0,n}=1$. Also, the `region' corresponding to the
reduced $(0,0)$ polytope is illustrated in Figure 3(c), which is the
$0$-simplex with $\widetilde{g}_{0,0}(x)=1$.

\begin{lemma}\label{lem1}
We have $g_{\alpha,\beta}(x)=\widetilde{g}_{m,n}(x)$.
\end{lemma}


\begin{proof}
We proceed by induction on $\alpha +\beta$.
When $\alpha =\beta =1$, $P_{1,1}$ is the $1$-simplex and therefore $g_{1,1}=1=\widetilde{g}_{0,0}$.
Also $g_{m+1,1}=g_{1,m+1}=1$ since $P_{m+1,1}$ is an $(m+1)$-simplex \cite[Example 3.16.8]{S} and by symmetry.
From the discussion of the initial cases above we know $\widetilde{g}_{m,0}=\widetilde{g}_{0,m}=1$ for all $m>0$.
Hence the result holds for $\alpha=1$ or $\beta=1$


Suppose that  $g_{a,b}=\widetilde{g}_{a-1,b-1}$ is true for all $a,b\ge 1$ and $2 \leq a+b \leq \alpha+\beta-1$.
We compute $f_{\alpha,\beta}$ using Definition \ref{def}(3).
Since the $P_{\alpha, \beta}$ is in $\mathbb{R}^{\alpha+\beta-1}$, the rank of the face lattice is $\alpha+\beta$ and hence $f_{\alpha,\beta}$ is a polynomial of degree $d=\alpha+\beta-1 (=m+n+1)$.
We will heavily use the fact that $P_{\alpha, \beta}$ is a pyramid with a basis of $\Delta_m\times \Delta_n$.
Note the fact that a face of $\Delta_m\times \Delta_n$ is of the form $\Delta_i\times \Delta_j$ for smaller $0\le i\le m$ or $0\le j\le n$.
That is, of the from $\Delta_{a-1}\times \Delta_{b-1}$ for smaller $1\le a\le m+1=\alpha$ or $1\le b\le n+1=\beta$, and there are $$C_{a,b}:={\alpha\choose a}{\beta\choose b}={m+1\choose a}{n+1\choose b}$$ faces of them.

There are two kinds of faces in $P_{\alpha, \beta}$:
\begin{enumerate}
\item Faces not containing the apex of the pyramid.
They are of the form $\Delta_{a-1}\times \Delta_{b-1}$ and each contributes $\widetilde{g}_{a-1,b-1}(x-1)^{d-(a+b-1)}$ to $f_{\alpha, \beta}$.
\item Faces containing the apex of the pyramid.
They are pyramid with a basis $\Delta_{a-1}\times \Delta_{b-1}$ and such a face will contributes $g_{a,b}(x-1)^{d-(a+b-1)}$ to $f_{\alpha, \beta}$.

\end{enumerate}
Hence a direct computation gives
\begin{align*}
\begin{split}
f_{\alpha,\beta} =& \left(\widetilde{g}_{m,n}+\sum_{a,b\ge 1, a+b=d} C_{a,b}\ g_{a,b}\right)(x-1)^{0}\\
& + \left(\sum_{a,b\ge 1, a+b=d} C_{a,b}\ \widetilde{g}_{a-1,b-1}+\sum_{a,b\ge 1, a+b=d-1} C_{a,b}\ g_{a,b}\right)(x-1)^{1}\\
& + \dots \\
&+ \left(\sum_{a,b\ge 1, a+b=3} C_{a,b}\ \widetilde{g}_{a-1,b-1}+\sum_{a,b\ge 1, a+b=2} C_{a,b}\ g_{a,b}\right)(x-1)^{d-2}\\
&+ \left(\sum_{a,b\ge 1, a+b=2} C_{a,b}\ \widetilde{g}_{a-1,b-1}+1\right)(x-1)^{d-1}+(x-1)^{d}
\end{split}
\end{align*}
Collecting the similar terms we have
\begin{align*}
\begin{split}
f_{\alpha,\beta}= & \widetilde{g}_{m,n}+
\left((x-1)^{d-1}+\sum_{k=2}^{d}\left(\sum_{a,b\ge 1, a+b=k} C_{a,b}\ g_{a,b}\right)(x-1)^{d-k}\right)\\
& + (x-1)\left((x-1)^{d-1}+\sum_{k=2}^{d}\left(\sum_{a,b\ge 1, a+b=k} C_{a,b}\ \widetilde{g}_{a-1,b-1}\right)(x-1)^{d-k}\right).
\end{split}
\end{align*}
By the induction hypothesis that $g_{a,b}=\widetilde{g}_{a-1,b-1}$
we reach at
\begin{equation}~\label{eq0}
f_{\alpha,\beta}= \widetilde{g}_{m,n}+x\left((x-1)^{d-1}+\sum_{k=2}^{d}\left(\sum_{a,b\ge 1, a+b=k} C_{a,b}\ \widetilde{g}_{a-1,b-1}\right)(x-1)^{d-k}\right).
\end{equation}

On the other hand, similarly by applying Definition \ref{def}(3) directly to compute $\widetilde{f}_{m,n}$ we will have
\begin{equation}\label{eq1}
\widetilde{f}_{m,n}=(x-1)^{d-1}+\sum_{k=2}^{d}\left(\sum_{a,b\ge 1, a+b=k} C_{a,b}\ \widetilde{g}_{a-1,b-1}\right)(x-1)^{(d-1)-(k-1)}.
\end{equation}

By comparing (\ref{eq0}) and (\ref{eq1}) we have
\begin{equation}\label{fgf}
f_{\alpha,\beta}=\widetilde{g}_{m,n}+x\widetilde{f}_{m,n}.
\end{equation}
Write $\widetilde{f}_{m,n}=h_{0}+h_{1}x+\cdots+h_{m+n}x^{m+n}$ and
$\widetilde{g}_{m,n}=h_{0}+(h_{1}-h_{0})x+\cdots+(h_{\ell}-h_{\ell-1})x^{\ell}$
with $\ell=\lfloor\frac{m+n}{2}\rfloor$. Therefore, by (\ref{fgf})
we obtain
$$f_{\alpha,\beta}=h_{0}+h_{1}x+\cdots+h_{\ell}x^{\ell}+h_{\ell}x^{\ell+1}+h_{\ell+1}x^{\ell+2}+\cdots+h_{m+n}x^{m+n+1}$$
and then by definition
\begin{equation*}
g_{\alpha,\beta}=
\left\lbrace
\begin{array}{ll}
h_{0}+(h_{1}-h_{0})x+\cdots+(h_{\ell}-h_{\ell-1})x^{\ell}+(h_{\ell}-h_{\ell})x^{\ell+1},   &\text{ if $m+n$ is odd;}\\
h_{0}+(h_{1}-h_{0})x+\cdots+(h_{\ell}-h_{\ell-1})x^{\ell},   &\text{ if $m+n$ is even.}
\end{array}
\right.
\end{equation*}
In either case, we have $g_{\alpha,\beta}=\widetilde{g}_{m,n}$ and
the lemma is proved.
\end{proof}

\section{The Proof of Theorem 5}
By Lemma 1 the remaining task for proving Theorem \ref{main2} is to
prove
$$\widetilde{g}_{m,n}(x)= \sum_{k=0}^{n}{m\choose k}{n\choose
k}x^{k}.$$

\textit{Proof of Theorem 5.} The proof is somewhat complicated and
here we sketch our strategy. We are to use induction on $m+n$. From
Definition~\ref{def}(2) it suffices to find
$\widetilde{f}_{m,n}(x)$. To do so, we will proceed as the
following.
\begin{enumerate}
\item (Subsection 4.1) Express $\widetilde{f}_{m,n}(x+1)$ (note the variable is
$x+1$ rather than $x$) in terms of $\widetilde{g}_{r,s}(x+1)$
with $r+s\le m+n-1$. It can be seen that
$\widetilde{f}_{m,n}(x+1)$ is a sum of entries, each multiplied
by a power of $x$, of a matrix $\widetilde{M}$ of size
$(m+n+1)\times (n+1)$.
\item  (Subsection 4.2) It turns out that it is more convenient to consider a modified matrix $\widehat{M}$ of size $(m+n+2)\times (n+1)$.
The reason to do so is that the corresponding sum
$\widehat{f}_{m,n}$, which is a Laurent polynomial, of entries of
$\widehat{M}$ is easier to compute.
\item (Subsection 4.3) From $\widehat{f}_{m,n}$ we compute $[x^r]\widehat{f}_{m,n}(x+1)$, the coefficient of
$x^r$ in $\widehat{f}_{m,n}(x+1)$.
\item (Subsection 4.4) From  $[x^r]\widehat{f}_{m,n}(x+1)$ we compute $[x^r]\widetilde{f}_{m,n}(x+1)$.
\item (Subsection 4.5) On the other hand, if our theorem is correct, then from Definition~\ref{def}(2) we can write $\widetilde{f}_{m,n}(x+1)$ in terms of
$\widetilde{g}_{r,s}(x+1)$, in which on each term Theorem 5 is
applied. Hence we obtain another expression of
$[x^r]\widetilde{f}_{m,n}(x+1)$
\item (Subsection 4.6) We check both expressions meet by way of a binomial identity. Hence the proof is completed.
\end{enumerate}

\subsection{The matrix $\widetilde{M}$}

We start with expressing $\widetilde{f}_{m,n}(x+1)$ in terms of $\widetilde{g}_{a-1,b-1}(x+1)$ for all $a,b\ge 1$ and $a+b \le n+m+1$.
Note that $0 \le a-1 \le m, 0 \le b-1 \le n$, and $n\leq m$.
The initial cases $\widetilde{g}_{0,0}(x+1)=\widetilde{g}_{1,0}(x+1)=\widetilde{g}_{0,1}(x+1)=1$ are proved in Lemma 1.
In what follows, we consider $1 \leq n \leq m$.
Writing equation (\ref{eq1}) and in terms of $m,n$ in the variable
$(x+1)$, we obtain
$$\widetilde{f}_{m,n}(x+1)=x^{m+n}+\sum_{k=2}^{m+n+1}\left(\sum_{a,b \ge 1, a+b=k}C_{a,b}\ \widetilde{g}_{a-1,b-1}(x+1)\right)x^{m+n+1-k}.$$

We define an $(m+n+1)\times(n+1)$ matrix $\widetilde{M}$ with the row
index $s$ from $0$ to $m+n$ (the top row index is $s=0$) and the
column index $t$ from $1$ to $n+1$ as follows:
\begin{enumerate}
\item The first column  is
$$(0,\dots,0,C_{m+1,1}\widetilde{g}_{m,0}(x+1),
C_{m,1}\widetilde{g}_{m-1,0}(x+1),\dots,C_{1,1}\widetilde{g}_{0,0}(x+1),1)^{\mathsf{T}},$$
where the first $n-1$ entries are $0$'s.
\item The last ($(n+1)$-th) column is
$$(C_{m,n+1}\widetilde{g}_{m-1,n}(x+1),C_{m-1,n+1}\widetilde{g}_{m-2,n}(x+1),\dots,C_{1,n+1}\widetilde{g}_{0,n}(x+1),0,\dots,0)^{\mathsf{T}},$$
where the last $(n+1)$ entries are $0$'s.
\item The $t$-th column, $2 \leq t \leq n$, is
$$(0,\dots,0,C_{m+1,t}\widetilde{g}_{m,t-1}(x+1), \dots,C_{1,t}\widetilde{g}_{0,t-1}(x+1),0,\dots,0)^{\mathsf{T}},$$ where the first
$n-t$ and the last $t$ entries are zeros.
\end{enumerate}

One can check that
\begin{equation}
\widetilde{f}_{m,n}(x+1)=\sum_{s=0}^{m+n}\sum_{t=1}^{n+1}\widetilde{M}_{s,t} x^{s},
\end{equation}
 where $\widetilde{M}_{s,t}$ is the
$(s,t)$ entry of $\widetilde{M}$.

For example, for $(m,n)=(4,3)$, the matrix $\widetilde{M}$ is
$$\widetilde{M}=\left[
\begin{array}{cccc}
  0 & 0 & C_{5,3}\widetilde{g}_{4,2}(x+1) & C_{4,4}\widetilde{g}_{3,3}(x+1)\\
  0 & C_{5,2}\widetilde{g}_{4,1}(x+1) & C_{4,3}\widetilde{g}_{3,2}(x+1) & C_{3,4}\widetilde{g}_{2,3}(x+1)\\
  C_{5,1}\widetilde{g}_{4,0}(x+1) & C_{4,2}\widetilde{g}_{3,1}(x+1) & C_{3,3}\widetilde{g}_{2,2}(x+1) & C_{2,4}\widetilde{g}_{1,3}(x+1)\\
  C_{4,1}\widetilde{g}_{3,0}(x+1) & C_{3,2}\widetilde{g}_{2,1}(x+1) & C_{2,3}\widetilde{g}_{1,2}(x+1) & C_{1,4}\widetilde{g}_{0,3}(x+1)\\
  C_{3,1}\widetilde{g}_{2,0}(x+1) & C_{2,2}\widetilde{g}_{1,1}(x+1) & C_{1,3}\widetilde{g}_{0,2}(x+1) & 0\\
  C_{2,1}\widetilde{g}_{1,0}(x+1) & C_{1,2}\widetilde{g}_{0,1}(x+1) & 0 & 0 \\
  C_{1,1}\widetilde{g}_{0,0}(x+1) & 0 & 0 & 0 \\
  1 & 0 & 0 & 0
\end{array}
\right]
$$

\subsection{The matrix $\widehat{M}$}
There are three type of the columns of $\widetilde{M}$, namely the
first, the last, and the middle columns. There is an additional $1$
in the first column, while there are only $m$ nonzero terms in the
last column. The middle columns all have $m+1$ nonzero terms.

To our end we would like to make them consistent, hence we modify it
into a $(m+n+2)\times(n+1)$ matrix $\widehat{M}$ as follows.
\begin{enumerate}
\item Add a `$-1$'-th row above the $0$-th row. Now the row index of $\widehat{M}$ goes from
$-1$ to $m+n$. In this $-1$-th row, the first $n$ entries are
all $0$'s and the $(n+1)$-th entry is
$C_{m+1,n+1}\widetilde{g}_{m,n}=\widetilde{g}_{m,n}$.
\item The entries
$\widetilde{M}_{m+n-(t-1),t}, 1 \leq t \leq n+1$ are $1,0,\dots
, 0$. Now define
$$\widehat{M}_{m+n-(t-1),t}:=C_{0,t}\widehat{g}_{-1,t-1},$$ where
$$\widehat{g}_{-1,t-1}:=\sum_{k=0}^{t-1}{-1\choose k}{t-1\choose k}(x+1)^{k}=\sum_{k=0}^{t-1}(-1)^{k}{t-1\choose k}(x+1)^{k}=(-x)^{t-1}$$
by applying Theorem \ref{main2} blindfoldedly. Now the list
 $1,0,\dots,0$ becomes
$${n+1\choose 1}(-x)^{0}, {n+1\choose 2}(-x)^{1},\dots, {n+1\choose n+1}(-x)^{n}.$$
\item For other indices $(s,t)$, let
$\widehat{M}_{s,t}=\widetilde{M}_{s,t}$ (i.e.
$\widehat{g}=\widetilde{g}$ for these indices).
\end{enumerate}

For example, for $(m,n)=(4,3)$, the matrix $\widehat{M}$ is
$$\widehat{M}=\left[
\begin{array}{cccc}
  \mathbf{0} & \mathbf{0} & \mathbf{0} & \mathbf{C_{5,4}\widetilde{g}_{4,3}(x+1)}\\
  0 & 0 & C_{5,3}\widetilde{g}_{4,2}(x+1) & C_{4,4}\widetilde{g}_{3,3}(x+1)\\
  0 & C_{5,2}\widetilde{g}_{4,1}(x+1) & C_{4,3}\widetilde{g}_{3,2}(x+1) & C_{3,4}\widetilde{g}_{2,3}(x+1)\\
  C_{5,1}\widetilde{g}_{4,0}(x+1) & C_{4,2}\widetilde{g}_{3,1}(x+1) & C_{3,3}\widetilde{g}_{2,2}(x+1) & C_{2,4}\widetilde{g}_{1,3}(x+1)\\
  C_{4,1}\widetilde{g}_{3,0}(x+1) & C_{3,2}\widetilde{g}_{2,1}(x+1) & C_{2,3}\widetilde{g}_{1,2}(x+1) & C_{1,4}\widetilde{g}_{0,3}(x+1)\\
  C_{3,1}\widetilde{g}_{2,0}(x+1) & C_{2,2}\widetilde{g}_{1,1}(x+1) & C_{1,3}\widetilde{g}_{0,2}(x+1) & \mathbf{{4\choose 4}(-x)^{3}}\\
  C_{2,1}\widetilde{g}_{1,0}(x+1) & C_{1,2}\widetilde{g}_{0,1}(x+1) & \mathbf{{4\choose 3}(-x)^{2}} & 0 \\
  C_{1,1}\widetilde{g}_{0,0}(x+1) & \mathbf{{4\choose 2}(-x)^{1}} & 0 & 0 \\
  \mathbf{{4\choose 1}(-x)^{0}} & 0 & 0 & 0
\end{array}
\right].
$$



We then can write entries of $\widehat{M}$ consistently.

\begin{lemma} For all $s,t$, we have
\begin{equation}\label{Mst}
\widehat{M}_{s,t}=C_{m+n+1-s-t,t}\cdot \widehat{g}_{m+n-s-t,t-1}(x+1),
\end{equation}
where $$\widehat{g}_{m+n-s-t,t-1}(x+1)=\sum_{k=1}^n {m+n-s-t-1\choose
k}{t-2\choose k} x^k$$ is from the formula of Theorem \ref{main2}.
\end{lemma}
\begin{proof}
First we look at the upper left $0$'s in $\widehat{M}_{s,t}$. These
indices are $1 \leq t \leq n-1$, $0 \leq s \leq (n-1)-t$, and $s=-1,
1 \leq t \leq n$. The equality (\ref{Mst}) holds since in the right
hand side $C_{m+n+1-s-t,t}={m+1 \choose m+n+1-s-t}{n+1\choose t}=0$
as $s+t<n$.

Next we look at the lower right $0$'s in $\widehat{M}_{s,t}$. These
indices are  $2 \leq t \leq n+1$, $m+n+2-t \leq s \leq m+n$. Again
the equality (\ref{Mst}) holds since in the right hand side
$C_{m+n+1-s-t,t}=0$ because $m+n+1-s-t<0$.

For the lower boundary, where $1 \leq t \leq n+1$ and $s=m+n-(t-1)$,
the equation~(\ref{Mst}) gives $C_{0,t}\widehat{g}_{-1,t-1}$, which
meets the definition of $\widehat{M}$.

For the rest entries, $\widehat{g}=\widetilde{g}$ from induction
hypothesis. Hecne the lemma is proved.
\end{proof}
%
%

Note that for those entries $\widehat{g}_{m+n-s-t,t-1}(x+1)$ with
$m+n-s-t<t-1$ (these are the cases $2 \leq t \leq n+1$ and $
m+n+2-2t \leq s \leq m+n-t$), the polynomial obtained is of degree
$m+n-s-t$ but we still view it as polynomials of degree $t-1$ with
zero coefficients for degrees greater than $m+n-s-t$.

\subsection{Coefficients of $\widehat{f}_{m,n}(x+1)$}
We define
$$\widehat{f}_{m,n}(x+1):=\sum_{s=-1}^{m+n}\sum_{t=1}^{n+1} \widehat{M}_{s,t} x^{s}.$$
Note that this is a Laurent polynomial containing terms of degrees
from $-1$ to $m+n$. Denote $[x^k]p(x)$ as the coefficient of $x^k$
in the Laurent polynomial $p(x)$. Our next task is the following.\\

\begin{lemma}~\label{fhat}
 We have
\begin{equation}\label{eq2}
[x^r]\widehat{f}_{m,n}(x+1)=\sum_{k=0}^{n}  {m+1\choose m+n-r-k}
\sum_{i=0}^{n-k} {n+1\choose i+k+1}\sum_{j=0}^{k}{i+j\choose
i}{m+n-r-k-1\choose i+j}{i+k\choose i+j}.
\end{equation}
\end{lemma}


We need another slight modification of $\widehat{M}$, that is, we can
extend the row index of $\widehat{M}$ into $s\le m+n$. It is because
that $C_{m+n+1-s-t,t}=0$ when $s \le -2, 1 \le t \le n+1$. Hence
from now on
$$\widehat{M}_{s,t}=C_{m+n+1-s-t,t}\widehat{g}_{m+n-s-t,t-1}(x+1), $$ with $s \leq
m+n$ and $1 \leq t \leq n+1$. The non-zero rows begins from $s=-1$.

\begin{lemma}
The nonzero $[x^r]x^s\widehat{M}_{s,t}$ comes from the terms with
indices $r-n\leq s \leq r$ and $r-s+1 \leq t \leq n+1$.
\end{lemma}

%
%
%
%

\begin{proof}

Fix $-1 \leq r \leq m+n$, it is clear that for $s>r$,
$[x^r]x^{s}\widehat{M}_{s,t}=0$ since the $\deg(\widehat{M}_{s,t})\ge 0$.
Similarly, $[x^r]x^{s}\widehat{M}_{s,t}=0$ for $s<r-n$ since
$\deg(\widehat{M}_{s,t})\le n$. Fix $r-n\leq s \leq r$, we only need to
look at entries $\widehat{M}_{s,t}$ for $t>r-s$. Also, because each
$\widehat{M}_{s,t}$ is viewed as a polynomial of degree $t$ for all $s$,
we obtain terms of degree $r$ from $x^{s}\widehat{M}_{s,t}$ for all
$r-n\leq s \leq r$ and $r-s+1 \leq t \leq n+1$.
%
\end{proof}
Hence, in $\widehat{M}$, such entries forms an isosceles right triangle
with side length $n+1$:
$$
\begin{array}{ccccccc}
  ~ & ~ &  ~ & ~ & ~ & ~ & \widehat{M}_{r-n,n+1} \\
  ~ & ~ &  ~ & ~ & ~ & * & * \\
  ~ & ~ &  ~ & ~ & * & * & * \\
  ~ & ~ &  ~ & \iddots & ~ & \vdots & \vdots \\
  ~ & ~ &  * & \dots & * & * & * \\
  ~ & *  & * & \dots & * & * & * \\
  \widehat{M}_{r,1} & *  & * & \dots & * & * & \widehat{M}_{r,n+1} \\
\end{array}
$$

\begin{proof} \emph{of Lemma 3:}
In the isosceles right triangle above, let the index $k$, running
from $0$ to $n$, indicates the hypotenuses top to down. Fix $k$, we
let the index $i$, running from $0$ to $n-k$ indicates the entry
from bottom to up on the $k$-th hypotenuse. Now we are to compute
\begin{eqnarray*}
[x^r]\widehat{f}_{m,n}(x+1) &=& \sum_{k=0}^n \sum_{i=0}^{n-k} [x^i] \widehat{M}_{r-i,i+k+1}.\\
&=& \sum_{k=0}^n \sum_{i=0}^{n-k} {m+1\choose m+n-r-k}{n+1\choose i+k+1}[x^i]g_{m+n-r-k-1, i+k}(x+1).
\end{eqnarray*}
Note that
\begin{eqnarray*}
[x^i]g_{m+n-r-k-1, i+k}(x+1) &=& [x^i]\sum_{\ell=0}^{i+k} {m+n-r-k-1\choose \ell} {i+k\choose l}(x+1)^{\ell}\\
&=& \sum_{j=0}^{k} {m+n-r-k-1\choose i+j} {i+k\choose i+j}{i+j\choose i}
\end{eqnarray*}
and the lemma is proved.
\end{proof}

\subsection{Coefficients of $\widetilde{f}_{m,n}(x+1)$}
We are ready to obtain $[x^r]\widetilde{f}_{m,n}(x+1)$. Recall that
$$\widetilde{f}_{m,n}(x+1)=\sum_{s=0}^{m+n}\sum_{t=1}^{n+1}x^{s}\widetilde{M}_{s,t},\qquad \widehat{f}_{m,n}(x+1)=\sum_{s=-1}^{m+n}\sum_{t=1}^{n+1}x^{s}\widehat{M}_{s,t}.$$
and $\widehat{M}$ is modified from $\widetilde{M}$, it is easy to have
$$\widetilde{f}_{m,n}(x+1)-x^{m+n}=\widehat{f}_{m,n}(x+1)-\sum_{k=0}^{n}x^{m+n-k}{n+1\choose k+1}(-x)^{k}-x^{-1}\widetilde{g}_{m,n}(x+1). $$

Since $\sum_{k=0}^{n}x^{m+n-k}{k+1 \choose n+1}(-x)^{k}=x^{m+n}$, by
the formula of $\widetilde{g}_{m,n}$ we obtain
$$\widetilde{f}_{m,n}(x+1)=\widehat{f}_{m,n}(x+1)-x^{-1}\sum_{k=0}^{n}{m\choose k}{n\choose k}(x+1)^{k}.$$
Hence
%
we obtain that $[x^r]\widetilde{f}_{m,n}(x+1)$ for $0\leq r \leq
m+n$ is given by
\begin{equation}\label{eq3}
\sum_{k=0}^{n}  {m+1\choose m+n-r-k}\sum_{i=0}^{n-k} {n+1\choose i+k+1} \sum_{j=0}^{k} {i+j\choose i} {m+n-r-k-1\choose i+j} {i+k\choose i+j} -{m\choose k}{n\choose k}{k\choose r+1}.
\end{equation}

\subsection{Another way}
In what follows, we will write $\widetilde{f}_{m,n}(x+1)$ in terms of $\widetilde{g}_{m,n}(x+1)$ given by the formula in Theorem \ref{main2} and expand it to read all the coefficients.
Recall that $\mathbf{S}_{\ell}:=\sum_{k=0}^{\ell}{m\choose k}{n\choose k}$ for $0 \leq \ell \leq \beta-1$.
Note that
$$
\widetilde{g}_{m,n}(x) = g_{\alpha, \beta}(x) = \sum_{k=0}^{\beta-1}{\alpha-1 \choose k}{\beta-1\choose k}x^k = \sum_{k=0}^{n}{m \choose k}{n\choose k}x^k.
$$
By using the facts that $\deg (\widetilde{f}_{m,n}(x))=m+n$,
coefficients of $\widetilde{f}$ are symmetric, and coefficients of
$[x^i]\widetilde{g}_{m,n}(x)= [x^{i+1}]\widetilde{f}_{m,n}(x)
-[x^i]\widetilde{f}_{m,n}(x) $, we have
\begin{align*}
\begin{split}
\widetilde{f}_{m,n}(x)=& (\mathbf{S}_{0}x^{0}+\mathbf{S}_{1}x^1+\cdots+\mathbf{S}_{n-1}x^{n-1})+(\mathbf{S}_{n}x^{n}+\mathbf{S}_{n}x^{n+1}\cdots+\mathbf{S}_{n}x^{m}) \\
&+(\mathbf{S}_{n-1}x^{m+1}+\cdots+\mathbf{S}_{0}x^{m+n}),
\end{split}
\end{align*}
hence
\begin{align*}
\begin{split}
\widetilde{f}_{m,n}(x+1)=& (\mathbf{S}_{0}(x+1)^{0}+\cdots+\mathbf{S}_{n-1}(x+1)^{n-1})+(\mathbf{S}_{n}(x+1)^{n}+\mathbf{S}_{n}(x+1)^{n+1}\cdots+\mathbf{S}_{n}(x+1)^{m}) \\
&+(\mathbf{S}_{n-1}(x+1)^{m+1}+\cdots+\mathbf{S}_{0}(x+1)^{m+n}).
\end{split}
\end{align*}


We are to compute $[x^r]\widetilde{f}_{m,n}(x+1)$. There are three
cases.

(i) $m+1 \leq r \leq m+n$. The coefficient
$[x^r]\widetilde{f}_{m,n}(x+1)$ comes from
$$\mathbf{S}_{m+n-r}(x+1)^{r}+\cdots+\mathbf{S}_{0}(x+1)^{m+n}.$$
Hence we have
\begin{eqnarray*}
[x^r]\widetilde{f}_{m,n}(x+1) &=& \sum_{i=0}^{m+n-r}\mathbf{S}_{i}{m+n-i\choose r}\\
 &=& \sum_{k=0}^{m+n-r} \sum_{i=k}^{m+n-r} {m\choose k}{n\choose k}{m+n-i\choose r}\\
&=& \sum_{k=0}^{m+n-r} \sum_{j=r}^{m+n-k} {m\choose k}{n\choose k}{j\choose r}\\
&=& \sum_{k=0}^{n}{m\choose k}{n\choose k}{m+n+1-k\choose r+1}.
\end{eqnarray*}

(ii) $n \leq r \leq m$. The coefficient
$[x^r]\widetilde{f}_{m,n}(x+1)$ comes from
$$(\mathbf{S}_{n}(x+1)^{r}+\cdots+\mathbf{S}_{n}(x+1)^{m})+(\mathbf{S}_{n-1}(x+1)^{m+1}+\cdots+\mathbf{S}_{0}(x+1)^{m+n}),$$
that is,
\begin{equation}\label{case2}
[x^r]\widetilde{f}_{m,n}(x+1)=\mathbf{S}_{n}\sum_{i=r}^{m}{k\choose
r}+\sum_{i=0}^{n-1}\mathbf{S}_{i}{m+n-i\choose r}.
\end{equation}
We claim that (\ref{case2}) equals to
$$\sum_{k=0}^{n}{m\choose k}{n\choose k}\sum_{j=r}^{m+n-k}{j\choose r}$$
as followms: for $0\le k\le n$ we collect terms with respect to
${m\choose k}{n\choose k}$. The first part of (\ref{case2})
contributes ${r\choose r}+\dots +{m\choose r}$, while the second
part contributes ${m+1\choose r}+\dots + {m+n-k\choose r}$.
 Hence
\begin{eqnarray*}
[x^r]\widetilde{f}_{m,n}(x+1) &=& \sum_{k=0}^{n}{m\choose k}{n\choose k}\sum_{j=r}^{m+n-k}{j\choose r}\\
 &=& \sum_{k=0}^{n}{m\choose k}{n\choose k}\sum_{j=r}^{m+n-k}{j\choose j-r}\\
&=&\sum_{k=0}^{n}{m\choose k}{n\choose k} {m+n+1-k\choose r+1}.
\end{eqnarray*}

(iii) $0 \leq r \leq n-1$. Now
$$[x^r]\widetilde{f}_{m,n}(x+1)=\sum_{i=r}^{n-1}\mathbf{S}_{i} {i\choose
r}+\sum_{i=n}^{m}\mathbf{S}_{n} {i\choose
r}+\sum_{i=0}^{n-1}\mathbf{S}_{i}{m+n-i\choose r} .$$

We collect terms with respect to ${m\choose k}{n\choose k}$. The
first summand contributes ${k\choose r}+\dots +{n-1\choose r}$ if
$0\le k\le r$, or ${k\choose r}+\dots +{n-1\choose r}$ if $r<k\le
n-1$. The second summand contributes ${n\choose r}+ \dots {m\choose
r}$, while the third summand contributes ${m+1\choose r}+\dots
+{m+n-k\choose r}$. In both two cases the sum can be combined into
$${m+n-k+1\choose r+1}-{k\choose r+1}$$ with $0\le k\le n-1$.
Hence
$$[x^r]\widetilde{f}_{m,n}(x+1)=\sum_{k=0}^n {m\choose k}{n\choose k} {m+n-k+1\choose r+1}- {m\choose k}{n\choose k} {k\choose r+1} $$

%
Note that $ {m\choose k}{n\choose k} {k\choose r+1}=0$ if $n\le r\le
m+n$. Hence in summary, we obtain that, for $0 \leq r \leq m+n$, the
coefficient of the term of degree $r$ is
\begin{equation}\label{eq4}
\sum_{k=0}^{n}{m\choose k}{n\choose k}{m+n-k+1\choose r+1}-{m\choose k}{n\choose k}{k\choose r+1}.
\end{equation}

\subsection{Both ends meet} The last piece of the proof is to compare
(\ref{eq3}) with (\ref{eq4}). We need the following generalization
of Vandermonde's identity:

\begin{lemma}\label{lem2} For all integers $m,n$ and  $q$, we have
$$
\sum_{k=0}^{n}\sum_{i=0}^{n-k}  \sum
_{j=0}^{k} {m+1\choose q-k+1}  {n+1\choose i+k+1} {i+j\choose j} {q-k\choose i+j}{k+i\choose k-j} =
\sum_{k=0}^{n} {m\choose k}{n\choose k}{m+n-k+1\choose q-k+1}
$$
\end{lemma}
Note that $m+n+1$ $q$'s from $-1$ to $m+n-1$ correspond to those
coefficients to be compared in the polynomial
$\widetilde{f}_{m,n}(x+1)$ of degree $m+n$. When $q=m+n$, we have
the Vandermonde identity. For other $q$'s both sides are zero.

\begin{proof}
We write the left hand side of the identity as $$ \sum_{k=0}^{n}
{m+1\choose q-k+1}\sum_{i=0}^{n-k}  {n+1\choose i+k+1} \sum
_{j=0}^{k} {i+j\choose j} {q-k\choose i+j}{k+i\choose k-j},$$ and
the proof is done by `summation from inside' once we have the
following three identities.
\begin{eqnarray}
 \sum
_{j=0}^{k} {i+j\choose j} {q-k\choose i+j}{k+i\choose
k-j}={k+i\choose k}{q\choose k+i},\\
\sum_{i=0}^{n-k}  {n+1\choose i+k+1}{k+i\choose k}{q\choose k+i}={q\choose k}{q+n-k+1\choose n-k}, \\
\sum_{k=0}^{n} {m+1\choose q-k+1}{q\choose k}{q+n-k+1\choose n-k}= \sum_{k=0}^{n} {m\choose k}{n\choose k}{m+n-k+1\choose q-k+1}.
\end{eqnarray}
The first two identities can be proved without too much difficulties
by using Vandermonde's identity. Here we give the proof of the third
one. From the left hand side (LHS) one has
\begin{eqnarray*}
LHS &=& \sum_{k=0}^{n}{m+1\choose q+1-k}{q\choose k} \sum_{j=0}^{n-k}{n\choose j}{q+1-k\choose n-k-j}\\
&=& \sum_{k=0}^{n} \sum_{j=0}^{n-k}{m+1\choose m-q+n-j}{m-q+n-j\choose n-k-j}{q\choose k}{n\choose j}\\
&=& \sum_{j=0}^{n} \sum_{k=0}^{n-j}{m+1\choose m-q+n-j}{m-q+n-j\choose n-k-j}{q\choose k}{n\choose j} = \star,
\end{eqnarray*}
where the second equality is from the fact
$${m+1\choose q+1-k}{q+1-k\choose n-k-j}={m+1\choose m-q+n-j}{m-q+n-j\choose n-k-j}.$$
Now
\begin{eqnarray*}
\star &=& \sum_{t=0}^{n}{m+1\choose q+1-t}{m+t\choose t}{n\choose t}\\
&=& \sum_{t=0}^{n} \sum_{k=0}^{t}{m+1\choose q+1-t}{m\choose k}{t\choose t-k}{n\choose t}\\
&=&  \sum_{t=0}^{n} \sum_{k=0}^{t}{m+1\choose q+1-t}{m\choose k}{n\choose n-k}{n-k\choose n-t}\\
&=& \sum_{k=0}^{n} \sum_{t=k}^{n}{n-k\choose n-t}{m+1\choose q+1-t}{m\choose k}{n\choose k}\\
&=& \sum_{k=0}^{n} \sum_{s=0}^{n-k}{n-k\choose s}{m+1\choose q+1-k-s}{m\choose k}{n\choose k}\\
&=& \sum_{k=0}^{n} \sum_{s=0}^{q+1-k}{n-k\choose s}{m+1\choose q+1-k-s}{m\choose k}{n\choose k}\\
&=& \sum_{k=0}^{n} {m+1+n-k\choose q+1-k}{m\choose k}{n\choose k}\\
&=& RHS.
\end{eqnarray*}

\end{proof}

Finally, let $q=m+n-1-r$ in Lemma \ref{lem2}. It can be seen that
(\ref{eq3}) is identical with (\ref{eq4}). In conclusion, we obtain
$$\widetilde{g}_{m,n}= \sum_{k=0}^{n}{m\choose k}{n\choose k}x^{k},$$
and the Theorem 5 is proved. \qed

\section{Concluding remarks}
%
%
To our knowledge there are very few polytopes with explicit toric
$f$- or $g$-polynomials computed. In this paper we give a new
example. Another possible approach, which we did not work on, is the
\emph{short toric polynomial} by Hetyei~\cite{H}. The short toric
polynomial contains the same information as Stanley's pair of toric
polynomials without the interwined definitions.

A curious observation is that, for the simplex $\Delta_m \times
\Delta_n$ with $m\ge n$ its Ehrhart $h^*$-polynomial is
$$h^*_{\Delta_m \times \Delta_n}=\sum_{k=0}^n{m-1\choose
k}{n-1\choose k}x^k,$$  which is only of slight difference from our
$g$-polynomial in Theorem 5. We hope our example can reveal the tip
of an iceberg hidden under.

An apparent generalization is to see if the lattice path matroid
polytope induced from a border strip still has a nice toric
$g$-polynoimial. However, initial observations reveal that the toric
$g$-polynomial of $P_{\alpha,\beta,\gamma}$, induced from the border
strip with $a$ boxes to the right, adding $b-1$ boxes up then $c-1$
boxes to the right, is \emph{not}
$$ g_{\alpha,\beta,\gamma}(x)=\sum_k{\alpha-1\choose k} {\beta-1\choose k} {\gamma-1\choose k}x^{k}. $$
It will be interesting if one can find other interesting families of
skew shapes $\lambda / \mu$ bounded by two lattice paths whose
corresponding $P_{\lambda / \mu}$ can have nice properties.

\section*{Acknowledgement}
The first named author would like to thank Sangwook Kim for
introducing this problem and Vic Reiner for discussions.



%

\end{document}